\documentclass[a4paper,12pt,leqno]{article}
\usepackage{amssymb}
\usepackage{amsxtra}
\usepackage{amsthm}
\usepackage[T1]{fontenc}
\usepackage{lmodern}
\usepackage[pagebackref,breaklinks]{hyperref}
\usepackage[lite,abbrev,shortalphabetic]{amsrefs}
\usepackage[shortlabels]{enumitem}
\usepackage{colonequals}
\theoremstyle{definition}
\newtheorem{defn}{Definition}[section]

\theoremstyle{plain}
\newtheorem{lemma}[defn]{Lemma}
\newtheorem{prop}[defn]{Proposition}

\newtheorem{theorem}[defn]{Theorem}
\numberwithin{equation}{section}
\setlist[enumerate,1]{(a)}

\newcommand{\Z}{\mathbb{Z}}
\newcommand{\Q}{\mathbb{Q}}
\newcommand{\R}{\mathbb{R}}

\begin{document}
\title{Fractional parts of powers of negative rationals}
\author{Qing Lu\thanks{School of Mathematical Sciences, Beijing Normal University, Beijing
100875, China; email: \texttt{qlu@bnu.edu.cn}.} \and Weizhe 
Zheng\thanks{Morningside Center of Mathematics, Academy of Mathematics and 
Systems Science, Chinese Academy of Sciences, Beijing 100190, China; 
University of the Chinese Academy of Sciences, Beijing 100049, China; email: 
\texttt{wzheng@math.ac.cn}.}}
\date{} \maketitle

\begin{abstract}
We prove that for any real number $\xi\neq 0$ and any coprime integers 
$p>q\ge1$ such that $\xi$ is irrational or $q>1$, the image in $\R/\Z$ of 
the sequence $(\xi (-p/q)^n)_{n\ge 0}$ is not contained in any interval of 
length less than $(1+q/p-q^2/p^2)/p$. 
\end{abstract}

\section{Introduction}
Distribution modulo one of geometric progressions has been studied 
extensively. One intriguing open question considered by Mahler \cite{Mahler} 
is the following: Does there exist a real number $\xi>0$ such that $\{\xi 
(3/2)^n\}<1/2$ for all $n\ge 0$? Here $\{x\}\colonequals x-\lfloor x\rfloor $ 
denotes the fractional part of a real number~$x$. Flatto, Lagarias, and 
Pollington \cite[Theorem 1.4]{FLP} proved that for any real number $\xi>0$ 
and coprime integers $p>q> 1$, we have 
\[
\limsup_{n\to \infty}\{\xi(p/q)^n\}
-\liminf_{n\to \infty}\{\xi(p/q)^n\}\ge 1/p.
\]
Dubickas (\cite[Theorem~1]{Dubickas-BLMS}, \cite[Theorem~2]{Dubickas}) proved 
a more general result which implies that for all real numbers $\xi\neq 0$ and 
$\eta$ and for any rational number $\alpha= p/q$ with integers $p>q\ge 1$ 
satisfying $\xi\notin \Q$ or $\alpha\notin\Z$, we have  
\begin{equation}\label{eq:Dubickas}
\limsup_{n\to \infty}\{\xi\alpha^n+\eta\} -\liminf_{n\to 
\infty}\{\xi\alpha^n+\eta\}\ge 1/p. 
\end{equation}
In fact, Dubickas' proof of \eqref{eq:Dubickas} also holds for $\alpha=-p/q$. 
In this paper, we improve the bound \eqref{eq:Dubickas} in the case where 
$\alpha$ is a negative rational number. 

\begin{theorem}\label{t:1}
Let $\xi\neq 0$ and $\eta$ be real numbers and let $p>q\ge 1$ be integers. 
Assume that $\xi\notin\Q$ or $p/q\notin \Z$. Then 
\begin{equation}\label{eq:t}
\limsup_{n\to \infty}\{\xi(-p/q)^n+\eta\} -\liminf_{n\to 
\infty}\{\xi(-p/q)^n+\eta\}\ge (1+r-r^2)/p, 
\end{equation}
where $r=q/p$.
\end{theorem}

While Dubickas proved various bounds for the limit values of 
$(\{\xi(-p/q)^n\})_{n\ge 0}$ \cite[Theorems~2, 4]{Dubickas-neg} and 
$(\lVert\xi(-p/q)^n\rVert)_{n\ge 0}$ \cite[Theorem~3]{Dubickas} (where 
$\lVert x\rVert=\min(\{x\},1-\{x\})$ denotes the distance to the nearest 
integer), no bound of the form \eqref{eq:t} better than $1/p$ was known to 
the best of our knowledge. In the case $(p,q)=(3,2)$ which is the negative 
analogue of Mahler's question, our bound is $(1+2/3-4/9)/3=11/27$. 

For a sequence $S=(s_n)_{n\ge 0}$ we let $S_n(x)=\sum_{i=0}^\infty s_{n+i} 
x^i$ denote the (shifted) generating series. We deduce Theorem \ref{t:1} from 
the following optimal bound. 

\begin{theorem}\label{t:2}
For any real number $0<r<1$ and any bounded sequence $S=(s_n)_{n\ge 0}$ of 
integers that is not ultimately periodic, we have 
\begin{equation}\label{eq:2}
\limsup_{n\to \infty} S_n(-r)-\liminf_{n\to \infty} S_n(-r)\ge 1+r-r^2.
\end{equation}
Moreover, there exists a bounded sequence $S$ of integers that is not 
ultimately periodic such that equality holds in \eqref{eq:2} for all $0<r<1$. 
\end{theorem}

Our proofs use techniques developed by Dubickas in \cite{Dubickas-BLMS} and 
\cite{Dubickas}.  One key fact that allows one to deduce Theorem \ref{t:1} 
from \eqref{eq:2} is that the sequence $(p\lfloor \xi (-p/q)^n 
+\eta\rfloor+q\lfloor \xi (-p/q)^{n+1} +\eta\rfloor)_{n\ge 0}$ of integers is 
not ultimately periodic. This was observed in various settings by Dubickas 
and Novikas (\cite[Lemma~2]{DN}, \cite[Theorem~3]{Dubickas-BLMS}, 
\cite[Lemma~1]{Dubickas}, \cite[Lemma~9]{Dubickas-neg}). The inequality 
\eqref{eq:2} is deduced from the Morse--Hedlund theorem on subword 
complexity. One difference with Dubickas' proof of \eqref{eq:Dubickas} is 
that we exploit the alternating sign in $S_n(-r)$, which leads to the optimal 
bound in Theorem \ref{t:2}. 

It follows immediately from Theorem \ref{t:1} that the image in $\R/\Z$ of 
the sequence $(\xi (-p/q)^n)_{n\ge 0}$ is not contained in any interval of 
length less than $(1+r-r^2)/p$. In subsequent work \cite{LZ}, we will examine 
the question whether the image can lie in an interval of length equal to 
$(1+r-r^2)/p$. In particular, we will show that the bound in Theorem 
\ref{t:1} is the best possible in the case $q=1$. The analogous question for 
powers of positive rational numbers was studied by several authors 
(\cite{FLP}, \cite{Bugeaud-linear}, \cite{BD}, \cite{Dubickas-small}, 
\cite{Dubickas-large}). 

\subsection*{Acknowledgments}
This work was partially supported by National Natural Science Foundation of 
China (grant numbers 12125107, 12271037, 12288201), Chinese Academy of 
Sciences Project for Young Scientists in Basic Research (grant number 
YSBR-033). 

\section{Proof of Theorem \ref{t:2}}

Our proof of \eqref{eq:2} uses the following consequence of the 
Morse--Hedlund theorem on subword complexity \cite[Theorem 7.3]{MH}. 

\begin{lemma}\label{l:MH}
Let $W$ be an infinite word over a finite alphabet that is not ultimately 
periodic and let $n\ge 0$. Then there exist a word $U$ of length $n$ and 
letters $s\neq t$ such that $sU$ and $tU$ are subwords of $W$. 
\end{lemma} 

\begin{proof}
This is a special case of \cite[Lemma~2]{Dubickas} or \cite[Corollary 
A.4]{Bugeaud}. We include a proof for the sake of completeness. Let $p_W(n)$ 
denote the number of subwords of $W$ of length $n$. By the Morse--Hedlund 
theorem, $p_W(n+1)>p_W(n)$. The assertion then follows from the pigeonhole 
principle. 
\end{proof}

We use Sturmian words in the proof of the last assertion of Theorem 
\ref{t:2}. Recall that an infinite word $s_0s_1\dots$ over the alphabet 
$\{0,1\}$ or a sequence $(s_0,s_1,\dots)$ with values in $\{0,1\}$ is 
\emph{Sturmian} if and only if there exists an irrational $\theta\in (0,1)$  
and a $\rho\in \R$ such that 
\begin{equation}\label{eq:S1}
s_n=\lfloor (n+1)\theta +\rho\rfloor -\lfloor 
n\theta +\rho\rfloor
\end{equation}
for all $n\ge 0$ or 
\begin{equation}\label{eq:S2}
s_n=\lceil (n+1)\theta +\rho\rceil -\lceil 
n\theta +\rho\rceil
\end{equation} 
for all $n\ge 0$. 

\begin{lemma}\label{l:Sturm}
Let $S=s_0s_1\dots$ be a Sturmian word over the alphabet $\{0,1\}$. Then $S$ 
is not ultimately periodic. Moreover, for all $m,n\ge 0$, the word 
$(s_{m}-s_n)(s_{m+1}-s_{n+1})\dots$ has no subword of the form $10^l1$ or 
$(-1)0^l(-1)$ for $l\ge 0$.  Furthermore, if $\sigma_n=\sum_{i=0}^{n-1} s_i$ 
denotes the partial sum, then, for all $m, n,k\ge 0$, we have 
$(\sigma_{m+k}-\sigma_{m})-(\sigma_{n+k}-\sigma_{n})\in \{-1,0,1\}$. 
\end{lemma}

Here, for a letter $a$, $a^l$ denotes the word $a\dots a$ of length $l$.

\begin{proof}
The first assertion is standard and the second and third assertions are 
equivalent forms of the fact that Sturmian words are balanced (see \cite{MH2} 
or \cite[Theorem 2.1.13]{Lothaire}). We include a proof for the sake of 
completeness. The first assertion follows from the identity 
\[\lim_{n\to \infty} \frac{\sum_{i=0}^{n-1} s_i}{n} = \lim_{n\to \infty} \frac{\lfloor n\theta+\rho\rfloor -\lfloor \rho\rfloor}{n}=\theta\]
and the observation that the limit is rational if $(s_n)_{n\ge 0}$ is 
ultimately periodic. The second and third assertions follow from the 
observation that, in the case \eqref{eq:S1}, we have 
\begin{multline*}
(\sigma_{m+k}-\sigma_{m})-(\sigma_{n+k}-\sigma_{n})=\sum_{i=0}^{k-1}(s_{m+i}-s_{n+i})\\
=(\lfloor (m+k)\theta +\rho\rfloor 
-\lfloor m\theta +\rho\rfloor)-(\lfloor (n+k)\theta +\rho\rfloor -\lfloor 
n\theta +\rho\rfloor)\\
=-(\{ (m+k)\theta +\rho\} -\{ m\theta +\rho\})+(\{ 
(n+k)\theta +\rho\} -\{ n\theta +\rho\})\\
\in (-2,2)\cap \Z= \{-1,0,1\}
\end{multline*}
and the same holds in the case \eqref{eq:S2} after replacing 
$\lfloor\cdot\rfloor$ by $\lceil\cdot\rceil$ and $\{\cdot\}$ by 
$-\{-(\cdot)\}$. 
\end{proof}

Let $S=(s_n)_{n\ge 0}$ be a sequence of integers. For $n,k\ge 0$, we have
\[S_n(x)-x^kS_{n+k}(x)=\sum_{i=0}^{k-1} s_{n+i}x^i.\]
Thus, for $n,m,k\ge 0$, we have
\begin{equation}\label{eq:triv} 
(S_n(x)-S_m(x))-x^k(S_{n+k}(x)-S_{m+k}(x))=\sum_{i=0}^{k-1} (s_{n+i}-s_{m+i})x^i.
\end{equation}
We will say that we apply \eqref{eq:triv} to the subwords 
$(s_n,\dots,s_{n+k-1})$ and $(s_m,\dots,s_{m+k-1})$. 

We now state and prove an analogue of Theorem \ref{t:2} for $S_n(r)$. 
Although not logically required, we hope that it makes the proof of Theorem 
\ref{t:2} easier to understand. 

\begin{prop}
For any real number $0<r<1$ and any bounded sequence $S=(s_n)_{n\ge 0}$ of 
integers that is not ultimately periodic, we have 
\begin{equation}\label{eq:p}
\limsup_{n\to \infty} S_n(r)-\liminf_{n\to \infty} S_n(r)\ge 1.
\end{equation}
Moreover, there exists a bounded sequence $S$ of integers that is not 
ultimately periodic such that equality holds in \eqref{eq:p} for all $0<r<1$.
\end{prop}

The inequality \eqref{eq:p} is a special case of \cite[Lemma~3]{Dubickas}. 

\begin{proof}
We include a proof of \eqref{eq:p} for the sake of completeness. Assume that 
the set of limit values of $(S_n(r))_{n\ge 0}$ is contained in an interval 
$(A,B)$ of length $<1$. Then, up to shifting (and truncating) $S$, we may 
assume that $S_n(r)\in (A,B)$ for all $n\ge 0$. Let $\mu =\max_{n\ge 0} s_n$ 
and $\nu=\min_{n\ge 0} s_n$. Since $S$ is not ultimately periodic, we have 
$\mu>\nu$. By \eqref{eq:triv}, 
\[2>(1+r)(B-A)>\mu-\nu.\] 
Thus $\mu-\nu=1$ and consequently $s_n\in \{\nu,\mu\}$ for all $n\ge 0$. We 
regard $S$ as a word $s_0s_1\dots$ over the alphabet $\{\nu,\mu\}$. Let $n\ge 
0$. By Lemma \ref{l:MH}, there exists a word $U=u_1\dots u_n$ of length $n$ 
such that $\mu U$ and $\nu U$ are subwords of $S$. By \eqref{eq:triv} applied 
to these subwords, 
\[(1+r^{n+1})(B-A)>\mu-\nu=1.\] 
Since $n\ge 0$ is 
arbitrary, we get $B-A\ge 1$. Contradiction. This finishes the proof of 
\eqref{eq:p}. 

Now let $S$ be a Sturmian sequence with values in $\{0,1\}$. Then for all 
$m,n\ge 0$, 
\[S_m(r)-S_n(r)=\sum_{i=0}^\infty (s_{m+i}-s_{n+i})r^i\le 1.\]
Here in the inequality we used the fact that the sequence 
$(s_{m+i}-s_{n+i})_{i\ge 0}$ takes values in $\{-1,0,1\}$ and alternates in 
sign if all zeroes are removed (Lemma \ref{l:Sturm}).  
\end{proof}

\begin{proof}[Proof of Theorem \ref{t:2}]
Assume that the set of limit values of $(S_n(-r))_{n\ge 0}$ is contained in 
an interval $(A,B)$ of length $<1+r-r^2$. Then, up to shifting (and 
truncating) $S$, we may assume that $S_n(-r)\in (A,B)$ for all $n\ge 0$. 

Let $\mu =\max_{n\ge 0} s_n$ and $\nu=\min_{n\ge 0} s_n$. Since $S$ is not 
ultimately periodic, we have $\mu>\nu$. Moreover, there exists a subword 
$(\mu,\mu')$ of $S$ with $\mu'<\mu$. Similarly, there exists a subword 
$(\nu,\nu')$ of $S$ with $\nu'>\nu$. If $\mu-\nu \ge 2$, then, by 
\eqref{eq:triv} applied to the subwords $(\mu,\mu')$ and $(\nu,\nu')$, 
\begin{multline*}
2-(1-r)(1-r^3)=(1+r^2)(1+r-r^2)\\
>(\mu-\nu)-r(\mu'-\nu')\ge (\mu-\nu)-r(\mu-\nu-2)\ge 2,
\end{multline*}
which is a contradiction. Thus $\mu-\nu=1$. Up to replacing $s_n$ by 
$s_n-\nu$, we may assume that $s_n\in \{0,1\}$ for all $n\ge 0$. For 
notational convenience, we now regard $S$ as a word $s_0s_1\dots$ over the 
alphabet $\{0,1\}$. By the above $01$ and $10$ are subwords of $S$. 

Note that $S$ has a subword of the form $10a$, where $a\in \{0,1\}$. If $010$ 
is a subword of $S$, then, by \eqref{eq:triv} applied to these subwords, we 
have 
\[1+r-r^2(1-r)(1-r^2)=(1+r^3)(1+r-r^2)>1+r+ar^2\ge 1+r,\]
which is a contradiction. Thus $010$ is not a subword of $S$. Similarly, 
$101$ is not a subword of $S$. 

Next note that $0111$ and $1000$ are not both subwords of $S$. Indeed, 
otherwise, by \eqref{eq:triv} applied to these subwords, we would have
\[1+r-r^2+r^3-r^3(1-r)(1-r^2)=(1+r^4)(1+r-r^2)>1+r-r^2+r^3.\] 
Up to replacing $s_n$ by $1-s_n$ for all $n\ge 0$, we may assume that $0111$ 
is not a subword of $S$. Then all blocks of zeroes in $S$ are separated by 
$11$. Thus $S=1^y0^{z_0}110^{z_1}110^{z_2}\dots$, where $y\ge 0$, $z_0\ge 1$, 
and $z_n\ge 2$ for all $n\ge 1$.

Assume that $z_n$ is odd for some $n\ge 1$. Then, by \eqref{eq:triv} applied 
to the subwords $10^{z_n}1$ and $0110^{z_n-1}$ of $S$, we have 
\[1+r-r^2+r^{z_n+1}-r^{z_n+1}(1-r)(1-r^2)=(1+r^{z_n+2})(1+r-r^2)>1+r-r^2+r^{z_n+1}.\]
Contradiction. 

Thus $z_n$ is even for all $n\ge 1$. Assume that $z_m-z_n>2$ for some $m,n\ge 
1$. Then $10^{z_n+3}$ and $0110^{z_n}1$ are both subwords of $S$. Thus, by 
\eqref{eq:triv} applied to these subwords, we have 
\[1+r-r^2+r^{z_n+3}-r^{z_n+3}(1-r)(1-r^2)=(1+r^{z_n+4})(1+r-r^2)>1+r-r^2+r^{z_n+3}.\]
Contradiction. 

Thus there exists an even integer $k\ge 2$ such that $z_n\in\{k,k+2\}$ for 
all $n\ge 1$. Since $S$ is not ultimately periodic, neither is the word 
$Z=z_1z_2\dots$. Let $n\ge 0$ be an integer. By Lemma \ref{l:MH} applied to 
$Z$, there exists a word $U=u_1\dots u_n$ of length $n$ over $\{k,k+2\}$ such 
that $(k+2)U$ and $kU$ are subwords of $Z$. Then 
$10^{k+2}110^{u_1}\dots110^{u_n}$ and $0110^{k}110^{u_1}\dots110^{u_n}$ are 
subwords of $S$.  Thus, by \eqref{eq:triv} applied to these subwords, we have 
\[(1+r^m)(B-A)>1+r-r^2,\]
where $m=k+3+\sum_{i=1}^n (u_i+2)> 4n$. Since $n\ge 0$ is arbitrary, we get 
$B-A\ge 1+r-r^2$. Contradiction. This finishes the proof of \eqref{eq:2}.

Let $k\ge 2$ be an even integer and let $W=w_0w_1\dots$ be a Sturmian word 
over the alphabet $\{0,1\}$. Let $z_n=k+2w_n$ for $n\ge 0$. Consider the word 
$S=0^{z_0}110^{z_1}11\dots$. We will show that 
\begin{equation}\label{eq:diff} 
S_j(-r)-S_l(-r)\le 1+r-r^2
\end{equation}
for all $j,l\ge 0$. 

For $n\ge 0$, let $\sigma_n=-1+\sum_{i=0}^{n-1} (z_i+2)=-1+(k+2)n+2\tau_n$, 
where $ \tau_n=\sum_{i=0}^{n-1}w_i$. Then $s_j=1$ if and only if $j=\sigma_n$ 
or $j=\sigma_n-1$ for some $n\ge 1$. For $n\ge 1$, we have 
\begin{gather}
\label{eq:Sr1}
\begin{aligned} S_{\sigma_n}(-r)&=1-\sum_{i=n+1}^\infty r^{\sigma_i-\sigma_n-1}(1-r)\\
 &>1-\sum_{i=n+1}^\infty r^{2(i-n)-1}(1-r)>1-\frac{r(1-r)}{1-r^2}=\frac{1}{1+r},
\end{aligned} 
 \\
\notag S_{\sigma_n-1}(-r)=\sum_{i=n}^\infty r^{\sigma_i-\sigma_n}(1-r)>0.
\end{gather}
Here in \eqref{eq:Sr1} we used the estimate $\sigma_i-\sigma_n\ge 
(k+2)(i-n)>2(i-n)$ for $i\ge n+1$. We have 
$S_{\sigma_n-1}(-r)=1-rS_{\sigma_n}(-r)<S_{\sigma_n}(-r)$ by \eqref{eq:Sr1}. 
Moreover, for $\sigma_{n-1}<j<\sigma_n$, 
$S_{j}(-r)=(-r)^{\sigma_n-j-1}S_{\sigma_n-1}(-r)$. Thus 
\[\max_{\sigma_{n-1}<j\le \sigma_n} S_j(-r)=S_{\sigma_n}(-r),\quad \min_{\sigma_{n-1}<j\le \sigma_n}S_j(-r)=S_{\sigma_n-2}(-r).\] 
Therefore, in order to show \eqref{eq:diff}, we may assume that $j=\sigma_m$ 
and $l=\sigma_n-2$ for some $m,n\ge 1$. In this case, 
\begin{align*}
S_{\sigma_m}(-r)-S_{\sigma_n-2}(-r)&=1+r-r^2-\sum_{i=1}^\infty (r^{\sigma_{m+i}-\sigma_m-1}-r^{\sigma_{n+i}-\sigma_n+1})(1-r)\\&\le 1+r-r^2.
\end{align*}
Here in the last inequality we used the inequality
\[(\sigma_{n+i}-\sigma_n)-(\sigma_{m+i}-\sigma_m)=2(\tau_{n+i}-\tau_n)-2(\tau_{m+i}-\tau_m)\ge -2,\]
which follows from Lemma \ref{l:Sturm}.
\end{proof}

\section{Proof of Theorem \ref{t:1}}

For $n\ge 0$, let
\begin{gather}
\notag x_n=\lfloor \xi (-p/q)^n+\eta\rfloor,\\
\label{eq:sn} s_n= -px_n-qx_{n+1}.
\end{gather}

We deduce Theorem \ref{t:1} from \eqref{eq:2} using the following lemma. 

\begin{lemma}\label{l:ap}
The sequence $(s_n)_{n\ge 0}$ is not ultimately periodic.
\end{lemma}

\begin{proof}
For $n\ge 0$, let 
\begin{multline*}
t_n=ps_{2n}-qs_{2n+1}=-p^2x_{2n}+q^2x_{2n+2}\\
=-p^2\lfloor \xi (p^2/q^2)^n+\eta\rfloor+q^2\lfloor \xi (p^2/q^2)^{n+1}+\eta\rfloor.
\end{multline*}
By \cite[Lemma~1]{Dubickas} applied to $\alpha=p^2/q^2$, $(t_n)_{n\ge 0}$ is 
not ultimately periodic. Assume that there exist $N\ge 0$ and $l\ge 1$ such 
that $s_{n+l}=s_n$ for all $n\ge N$. Then $t_{n+l}=t_n$ for all $n\ge N/2$. 
Contradiction. 
\end{proof}

\begin{proof}[Proof of Theorem \ref{t:1}]
For $n\ge 0$, let 
\[y_n=\{ \xi (-p/q)^n+\eta\}-\eta.\]
Since $x_n+y_n=\xi(-p/q)^n$, we have
\begin{equation}\label{eq:sn2}
s_n = py_n+qy_{n+1}.
\end{equation}
By \eqref{eq:sn} and \eqref{eq:sn2}, $(s_n)_{n\ge 0}$ is a bounded sequence 
of integers. Thus, by Lemma \ref{l:ap} and Theorem \ref{t:2}, \eqref{eq:2} 
holds. Since 
\[S_n(-r)=\sum_{i=0}^\infty s_{n+i} (-r)^i=\sum_{i=0}^\infty p(y_{n+i}+ry_{n+i+1})(-r)^i=py_n, \]
\eqref{eq:2} is equivalent to
\[\limsup_{n\to \infty} y_n -\liminf_{n\to \infty} y_n \ge (1+r-r^2)/p\]
and hence to \eqref{eq:t}.
\end{proof}

\end{document}